\documentclass[12pt]{article}

\usepackage{amsfonts,amsmath,amsthm}\usepackage{latexsym,color,epsfig,mathrsfs, enumerate}

\newtheorem{thm}{Theorem}
\newtheorem{lem}[thm]{Lemma}

\newtheorem{prop}[thm]{Proposition}

\theoremstyle{definition}
\newtheorem{dfn}[thm]{Definition}
\newtheorem{rem}[thm]{Remark}

\newcommand{\mc}[1]{\mathcal{#1}}
\newcommand{\mb}[1]{\mathbb{#1}}

\newcommand{\sub}{\subset}

\newcommand{\sm}{\setminus}
\newcommand{\es}{\emptyset}

\newcommand{\ov}{\overline}

\newcommand{\eps}{\varepsilon}

\newcommand{\dD}{\delta}

\newcommand{\DD}{\Delta}
\newcommand{\OO}{\Omega}

\title{Distinct degrees in induced subgraphs}

\author{Matthew Jenssen\thanks{Mathematical Institute, University of Oxford, Oxford, UK. E-mail: 
\texttt{jenssen@maths.ox.ac.uk.}} \and Peter Keevash\thanks{Mathematical Institute, University of Oxford, Oxford, UK. E-mail: 
\texttt{keevash@maths.ox.ac.uk.}} \and Eoin Long\thanks{School of Mathematics, University of Birmingham, Birmingham, UK. E-mail: 
\texttt{e.long@bham.ac.uk.}} \and Liana Yepremyan\thanks{Mathematical Institute, University of Oxford, Oxford, UK. E-mail: 
\texttt{yepremyan@maths.ox.ac.uk.}\newline\hspace*{1.5em} 
Research supported in part by ERC Consolidator Grant 647678.}}

\begin{document}

 \maketitle

\abstract{
An important theme of recent research in Ramsey theory
has been establishing pseudorandomness properties of Ramsey graphs.
An $N$-vertex graph is called {\em $C$-Ramsey}
if it has no homogeneous set of size $C\log N$.
A theorem of Bukh and Sudakov, 
solving a conjecture of  Erd\H{o}s, Faudree and S\'os,
shows that any $C$-Ramsey $N$-vertex graph
contains an induced subgraph with $\Omega_C(N^{1/2})$ distinct degrees.
We improve this to $\Omega_C(N^{2/3})$,
which is tight up to the constant factor.

We also show that any $N$-vertex graph with $N > (k-1)(n-1)$ 
and $n\geq n_0(k) = \Omega (k^9)$ either contains a homogeneous set 
of order $n$ or an induced subgraph with $k$ distinct degrees. 
The lower bound on $N$ here is sharp, as shown by an appropriate Tur\'an graph, and confirms a conjecture of Narayanan and Tomon.
} 

\section{Introduction}

A major open problem in Ramsey theory is the construction
of explicit graphs that are approximately tight for Ramsey's theorem;
all known constructions involve some randomness, 
which motivates a substantial literature establishing
that Ramsey graphs have certain pseudorandomness properties.
Given a graph $G$, we call $U \sub V(G)$ {\em homogeneous}
if the induced subgraph $G[U]$ is complete or empty.
Ramsey's theorem states that $\hom (G) \to \infty $ 
as $N := |V(G)| \to \infty$. In a more quantitative form, 
we have $\tfrac {1}{2} \log _2N \le \hom (G) \le 2 \log _2N$,
where the lower bound is due to Erd\H{o}s and Szekeres \cite{ESR}
and the upper bound to Erd\H{o}s \cite{E} (the birth
of the Probabilistic Method in Combinatorics).
It is remarkable that in the 70+ years since these results
there have only been improvements to the lower order terms
(see the survey \cite{CFS}). Furthermore, there is no known
explicit construction of an $N$-vertex graph $G$ with
$\hom(G) = O(\log N)$, despite intense interest in this question
and the related notions of randomness extraction/dispersion in
Computer Science; the best known explicit construction 
due to Li \cite{Li} gives 
$\hom(G) = (\log N)^{O(\log\log\log N)}$.

Motivated by both the difficulty in providing explicit constructions
and the challenge in improving the bounds for the Ramsey problem,
an important theme of recent research in Ramsey theory
has been establishing properties of Ramsey graphs 
supporting the intuition that they should be `random-like'.
This indirect study has been very fruitful, and it is now known 
that $N$-vertex Ramsey graphs display similar behaviour  
to the Erd\H{o}s-Renyi random graph $G_{N,1/2}$ in many respects:\ 
the edge density by Erd\H{o}s and Szemer\'edi \cite{ES}; 
universality of small induced subgraphs by Pr\"omel and R\"odl \cite{PR}; 
the number of non-isomorphic induced subgraphs by Shelah \cite{Sh}; 
the sizes and orders of induced subgraphs 
by Kwan and Sudakov \cite{KS, KS2} 
and Narayanan, Sahasrabudhe and Tomon \cite{NST}.

Here we consider a problem of
Erd\H{o}s, Faudree and S\'os \cite{EFS}
concerning induced subgraphs with many distinct degrees.
Given a graph $G$, we let 
	\begin{equation*}
		f(G) := 	\max \big \{k\in {\mathbb N}: G \mbox{ has an induced subgraph with } k \mbox { distinct degrees}\big \}.
	\end{equation*}
Bukh and Sudakov \cite{BS} showed that any $N$-vertex graph $G$ 
with $\hom (G) \leq C \log N$ has $f(G) = \Omega _C(N^{1/2})$,
thus confirming a conjecture in \cite{EFS}
motivated by the observation that
$f\big ( G_{N,1/2} \big ) = \Omega (N^{1/2})$
with high probability (whp);
they noted however the lack of a corresponding upper bound,
and showed that whp $f\big ( G_{N,1/2} \big ) = O(N^{2/3})$. 
An unpublished result of Conlon, Morris, Samotij and Saxton \cite{CMSS} 
shows that whp $f\big ( G_{N,1/2} \big ) = \Omega(N^{2/3})$,
so this in fact gives the correct order.
Our first theorem establishes the same lower bound 
for Ramsey graphs, which is therefore tight up to the constant factor. 

\begin{thm}
	\label{theorem: D_D_Ramsey}
	Let $G$ be an $N$-vertex $C$-Ramsey graph. Then 
	$f(G) = \Omega _C \big ( N^{2/3} \big )$.
\end{thm}

Moreover, we establish this lower bound on $f(G)$ using only
the combinatorially simpler `diversity' property (see \cite{Sh,BS})
that many vertices have dissimilar neighbourhoods: 
we say $U \sub V(G)$ is {\em $\dD$-diverse}
if $|N_G(u) \triangle N_G(u')| \geq \delta |V(G)|$
for any distinct $u,u'$ in $U$.

\begin{thm}
\label{theorem: distinct_degrees_Ramsey}
Given $\delta > 0$ there is $c>0$ 
such that any $N$-vertex graph $G$ 
with 
a $\dD$-diverse set of size $N^{2/3}$ has 
an induced subgraph with at least $cN^{2/3}$ distinct degrees.
\end{thm}

Theorem \ref{theorem: distinct_degrees_Ramsey} 
implies Theorem \ref{theorem: D_D_Ramsey}
as the hypotheses of the former
follow from those of the latter
by results of Kwan and Sudakov \cite{KS2}
(see subsection \ref{subsec:pfs}).

It is also natural to investigate the relationship 
between $\hom (G)$ and $f(G)$ in more generality.
Narayanan and Tomon \cite{NT} showed 
for any $k \in {\mathbb N}$, $\eps >0$
and $N \geq N_0(k,\eps )$ that 
any $N$-vertex graph $G$ has $f(G) \ge k$
or $\hom (G) \geq N / (k-1+\eps)$.
They conjectured that the optimal relationship
between $\hom (G)$ and $f(G)$ when $|V(G)| \gg f(G)$
should be given by the $(k -1)$-partite Tur\'an graph 
on $N = (k -1)(n-1)$ vertices, which has 
$f(G)=k-1$ and $\hom(G)=n-1 = N/(k-1)$.
We confirm this conjecture, thus obtaining an exact result,
and moreover we only require a lower bound on $n$ 
that is polynomial in $k$
(in \cite{NT} an exponential lower bound is assumed).

\begin{thm}
	\label{theorem: distinct_degrees_large_homogen}
	Suppose $G$ is an $N$-vertex graph with 
	$N> (n-1)(k-1)$, where $n = \Omega ({k^{9}})$. 
	Then $f(G) \geq k$ or $\hom (G) \geq n$.
\end{thm}

We prove Theorems \ref{theorem: D_D_Ramsey}
and \ref{theorem: distinct_degrees_Ramsey} 
in the next section and Theorem
\ref{theorem: distinct_degrees_large_homogen}
in the following section.
The final section contains some concluding remarks.

\section{Distinct degrees in Ramsey graphs}

Our proof that any sufficiently diverse graph
contains an induced subgraph with many distinct degrees
naturally splits into two pieces. 

In the first subsection we give a new perspective:
we reduce the problem to a continuous relaxation
(in a similar spirit to \cite[Section 3]{KL}).
We show that it is sufficient to define
a probability distribution on the vertex set,
with respect to which a random induced subgraph
has a large set of vertices whose 
expected degrees are well-separated.

While this change of perspective creates a larger
and more flexible solution space, the existence
of the required distribution is still quite subtle.
In the second subsection we show its existence
via an additional randomisation, in which 
the probabilities themselves are randomly generated
according a distribution that takes into account
the neighbourhood structure of our graph.

In the final subsection of this section 
we combine the two above ingredients 
to prove our result on diverse graphs
(Theorem \ref{theorem: distinct_degrees_Ramsey})
and deduce (via results of Kwan and Sudakov)
our result on Ramsey graphs
(Theorem \ref{theorem: D_D_Ramsey}).

\subsection{A continuous relaxation}

Let $G$ be a graph with vertex partition $V(G) = U \cup V$.
Given  ${\bf p} = (p_v)_{v\in V} \in [0,1]^V$, 
let $G({\bf p}) = G[U \cup W]$ denote 
the random induced subgraph 
where $W$ contains each $v \in V$
independently with probability $p_v$. 
The main result of this subsection is the following lemma,
showing that separation of expected degrees in $G({\bf p})$
guarantees an induced subgraph with distinct degrees.

\begin{lem}
	\label{lem: switch to probabilistic world}
Given $\delta >0$ there is $c > 0$ so that the following holds. Let $G$ be a graph with vertex partition 
$V(G) = U \cup V$ where $|V|=N$. 
Suppose also that $U' \subset U$ 
and ${\bf p} \in [0.1,0.9]^V$ such that
any distinct $u,u'$ in $U'$ satisfy
\[ \big | {\mathbb E} \big ( d_{G({\bf p})}(u) \big ) 
	- {\mathbb E} \big ( d_{G({\bf p})}(u') \big ) 
	\big | \geq \delta 
	\ \text{ and } \
	\big | \big (N_G(u) \triangle N_G(u') \big ) \cap V\big | 
	\geq \delta N. \] 
Then there is $W \subset V$ so that $G[U \cup W]$ 
has at least $c |U'|$ distinct degrees.
\end{lem}

The idea of the proof is that in $G({\bf p})$
a vertex typically has degree within $O(\sqrt{N})$
of its expectation, and if we restrict 
to the set $B$ of such `balanced' vertices 
then a pair of vertices $u,u' \in U'$ can only have equal degrees 
when their expected degrees differ by $O(\sqrt{N})$.
The separation of expected degrees implies that
$B$ has only $O_{\delta }(|U'|\sqrt{N})$ such pairs.
Each has equal degrees with probability $O_{\delta }(1/\sqrt{N})$,
by diversity and an anti-concentration estimate,
so we can ensure that $B$ has only $O_{\delta }(|U'|)$
pairs with equal degree in $U'$; then Tur\'an's theorem
will provide the required conclusion.
The required anti-concentration estimate 
is the following generalisation of the well-known
Erd\H{o}s-Littlewood-Offord inequality \cite{E2};
this is not a new result, but for completeness and
the convenience of the reader we will give a simple 
deduction from \cite{E2}, namely the case that all $p_i=1/2$. 

\begin{prop} 
\label{prop: anticoncentration bound}
Fix non-zero reals $a_1,\dots,a_n$
and $p_1,\dots,p_n$ in $[0.1,0.9]$.
Let $X_1,\ldots, X_n$ be independent 
Bernoulli random variables with $X_i \sim Be(p_i)$,
i.e.\ $\mb{P}(X_i=1)=p_i$ and $\mb{P}(X_i=0)=1-p_i$. Then 
$\max _{x\in {\mb R}} {\mb P}(\sum _{i\in [n]} a_i X_i = x) = O(n^{-1/2})$.
\end{prop}

\begin{proof}
For each $i$ we fix $w_i,z_i \in [0,1]$
with $p_i = w_i/2 + (1-w_i)z_i$ 
and write $X_i = W_i Y_i + (1-W_i)Z_i$, 
where $Y_i \sim Be(1/2)$, $W_i \sim Be(w_i)$ 
and $Z_i \sim Be(z_i)$ are independent.
We make this choice so that each $w_i \ge 0.2$:
e.g.\ if $p_i \le 1/2$ let $z_i=0$ and $w_i=2p_i$,
or if $p_i > 1/2$ let $z_i=1$ and $w_i=2(1-p_i)$.
We condition on any choice $C$ 
of the $W_i$'s and $Z_i$'s,
which determines $I := \{i: W_i=1\}$. 
By Chebyshev's inequality,
$\mb{P}(|I|<n/10) < O(n^{-1})$,
so it suffices to bound
${\mb P}(\sum _{i\in [n]} a_i X_i = x \mid C)$
for any $C$ such that $|I| \ge n/10$;
the required bound $O(n^{-1/2})$ holds 
by \cite{E2} applied to $(Y_i: i \in I)$.
\end{proof}

We also use of the following version of Tur\'an's theorem (see e.g. Chapter 6 in \cite{Bol}).

\begin{thm}
	Any $n$-vertex graph $G$ with average degree $d$ contains 
	an independent set of size at least $n/(d+1)$.
\end{thm}

\begin{proof}[Proof of Lemma \ref{lem: switch to probabilistic world}]
Note that we can assume $N$ is large, by taking $c>0$ small enough. Let $H$ be a random induced subgraph 
according to $G({\bf p})$ and
\begin{align*}
B & = \{ u \in U':
 \big |d_H(u) - {\mathbb E}\big (d_{G({\bf p})}(u) \big )\big | 
		\leq \sqrt{N} \}, \\
P & = \big \{ \{u,u'\} \subset U': \big | 
		{\mathbb E}\big (d_{G({\bf p})}(u) \big ) 
		- {\mathbb E}\big (d_{G({\bf p})}(u') \big )\big | \leq 
		2\sqrt{N} \big \}, \text{ and } \\
J & = \{ \{u,u'\} \in P: d_H(u) = d_H(u') \}.
\end{align*}
We claim that with positive probability
we have both $|B| \ge |U'|/2$ and $|J| = O_\dD(|U'|)$.
This claim implies the lemma, as by Tur\'an's theorem
$J[B]$ contains an independent set of size $\Omega_\dD(|U'|)$,
which must consist of vertices with distinct degrees, 
as if $u,u'$ are in $B$ and $d_H(u)=d_H(u')$
then $\{u,u'\} \in P$, so $\{u,u'\} \in J$.

To prove the claim, we first estimate $|B|$.
For any $u \in U'$ we have 
	${\mathbb V}ar \big (d_{G({\bf p})}(u) \big ) 
	\leq \sum _{v\in V} p_v(1-p_v) \leq N/4$ 
and so Chebyshev's inequality gives 
${\mathbb P}(u \notin B) \leq \frac {N/4}{\sqrt{N}^2} = 1/4$. 
Thus ${\mathbb E}(|U'\sm B|) \leq  |U'|/4$,
so by Markov's inequality 
${\mathbb P}\big (|U' \sm B| \geq |U'|/2 \big ) \leq 1/2$,
i.e.\ ${\mathbb P}\big (|B| \geq |U'|/2 \big ) \geq 1/2$.

To estimate $|J|$, we first note that 
by the degree separation property 
we have $|P| \leq 2\dD^{-1}|U'|N^{1/2}$. 
Each $\{u,u'\} \in P$ belongs to $J$ with probability 
${\mathbb P}\big (d_H(u) - d_H(u') = 0 \big ) = O((\dD N)^{-1/2})$ 
by Proposition \ref{prop: anticoncentration bound},
which can be applied by the diversity property
and the assumption that all $p_v \in [0.1,0.9]$.
Thus $\mb{E}|J| = O_\dD(|U'|)$, 
so $\mb{P}\big (|J| = O_\dD(|U'|) \big )>1/2$
by Markov's inequality.
This proves the claim and so the lemma.
\end{proof}

\subsection{Solving the relaxation in diverse graphs}

The following lemma shows how to find
the distribution ${\bf p}$ required to apply 
Lemma \ref{lem: switch to probabilistic world}.

\begin{lem}
	\label{lem: constructing probability vector}
Given $\delta > 0$ there is $c>0$ such that the following holds. Let $G$ be a graph with vertex partition $V(G) = U \cup V$ where $U$ is $\dD$-diverse, $|U| \le N^{2/3}$ and 
$|V|=N$. Then there are ${\bf p} \in [0.1, 0.9]^V$ 
and $U' \subset U$ with $|U'| \geq c |U|$ so that
	$| {\mathbb E}\big ( d_{G({\bf p})}(u) \big ) - 
	{\mathbb E}\big ( d_{G({\bf p})}(u') \big ) | \geq 1$
for all distinct  $u,u' \in U'$.
\end{lem}

The key idea is that our construction of the 
probability vector ${\bf p}$ is itself random,
with a distribution depending on
the neighbourhood structure of $G$.
We start by sketching 
a simplified proof of the lemma 
under the stronger assumption $|U| = O(N^{2/3}/\log ^{1/3}N)$.
For each $u \in U$ we define a `signed neighbourhood vector' 
${\bf u} \in \{-1,1\}^V$ by 
${\bf u}_v = 1$ if $uv \in E(G)$
or ${\bf u}_v = -1$ otherwise. 
Let ${\bf 1} \in [0,1]^V$ denote the `all-1' vector.
We randomly select integers 
$m_u \in \big [- |U|, |U| \big ]$ 
uniformly and independently for all $u \in U$.
Then we let
\begin{equation}
	\label{equation: probability choice}
	{\bf p} := \frac {1}{2} {\bf 1}  
	+ \sum _{u\in U} \Big ( \frac {m_u}{N} \Big ) {\bf u}.
\end{equation}
The variance of each coordinate $p_v$ of $\bf p$ is at most $|U|^3/N^2 = O(\log N)^{-1}$ and so, by a standard concentration argument, with high probability
${\bf p} \in [0.1,0.9]^V$
(for an appropriate choice of the implicit 
constant in the stronger assumption on $|U|$).
On the other hand, our definition of ${\bf p}$
in terms of the neighbourhood structure 
relates expected degree differences 
to our diversity assumption, as follows.
For any distinct $u,u'$ in $U$, as
${\mathbb E}\big ( d_{G({\bf p})}(u) \big ) = 
d_{G[U]}(u) + ({\bf 1} +  {\bf u}) \cdot {\bf p}/2$,
we have
\begin{equation}
	\label{equation: dependence of p}
	 {\mathbb E}\big ( d_{G({\bf p})}(u) \big )  - 
	{\mathbb E}\big ( d_{G({\bf p})}(u') \big ) = 
	d_{G[U]}(u) - d_{G[U]}(u') + ({\bf u} - {\bf u}') 
	\cdot {\bf p}/2 .
\end{equation}
Let ${\cal E}_{u,u'}$ denote the event that 
	$|{\mathbb E}\big ( d_{G({\bf p})}(u) \big ) - 
	{\mathbb E}\big ( d_{G({\bf p})}(u') \big )| \leq 1$.
Conditional on any choice of ${\bf m} = (m_w)_{w \ne u}$,
we see from \eqref{equation: probability choice}
and \eqref{equation: dependence of p} 
that there is some interval $I$ of length $4$
(depending on ${\bf m}$) such that
${\cal E}_{u,u'}$ holds if and only if
$({\bf u} - {\bf u}') \cdot 
\tfrac{m_u}{N} {\bf u} \in I$.
As $({\bf u} - {\bf u}') \cdot {\bf u} 
=  2\big |\big (N_G(u)\triangle N_G(u')\big ) \cap V\big |
\ge 2\dD N$, this corresponds to a choice of $m_u$
in an interval of length $O(\dD^{-1})$,
which occurs with probability $O(\dD^{-1}|U|^{-1})$.
By Markov's inequality, we can therefore choose 
${\bf p} \in [0.1,0.9]^V$ so that only
$O(\dD^{-1}|U|)$ such ${\cal E}_{u,u'}$ hold.
Then by Tur\'an's theorem there is $U' \sub U$
of size $\OO(\dD|U|)$ within which
no such events ${\cal E}_{u,u'}$ hold,
as required.

The actual proof is similar to the above sketch,
except that we cannot rely on concentration
of measure to ensure ${\bf p} \in [0.1,0.9]^V$;
instead, we `truncate the outliers'.

\begin{proof}[Proof of Lemma \ref{lem: constructing probability vector}]
By taking $c$ small enough we may assume that $N \geq N_0(\delta )$. Secondly, replacing $U$ with a subset if necessary, we can also assume that $|U| \leq \delta N^{2/3}/5$. Let $(m_u)_{u\in U}$ and ${\bf p}$ be as in \eqref{equation: probability choice}.
For $u \in U$ we write
${\bf q}^u = {\bf p} - \tfrac{m_u}{N} {\bf u}$,
and note that ${\bf q}^u$ is independent of $m_u$. 
We call $u$ {\em good} if there are at most $\dD N/2$
coordinates $v \in V$ with $q^u_v \notin [0.2,0.8]$, and bad otherwise. We also write $U^g$ for the set of good vertices in $U$.

We claim that $\mb{P}(|U^g| \ge |U|/2) > 1/2$.
To see this, we note for any $u$ and $v$ that
$q^u_v - 1/2 = N^{-1} \sum_{u' \ne u} \pm m_{u'}$
is a random variable with mean $0$
and variance at most $N^{-2} |U|^3 < 0.01\dD$,
so by Chebyshev's inequality
$\mb{P}(|q^u_v - 1/2|>0.3) 
< 0.01 \dD/0.3^2 = \dD/9$.
Thus the expected number of $v$
with $q^u_v \notin [0.2,0.8]$
is at most $\dD N/9$, so by Markov's inequality
$u$ is bad with probability less than $1/4$.
Now the expected number of bad $u$
is less than $|U|/4$, so by Markov's inequality
more than half of $U$ is bad
with probability less than $1/2$.
The claim follows.

Now we define ${\bf p}' \in [0.1,0.9]^V$ 
by truncating ${\bf p}$: for each $v \in V$, 
if $p_v<0.1$ let $p'_v=0.1$,
if $p_v>0.9$ let $p'_v=0.9$,
or let $p'_v=p_v$ otherwise.
We let ${\cal E}_{u,u'}$ denote the event that 
$|{\mathbb E}\big ( d_{G({\bf p}')}(u) \big ) - 
 {\mathbb E}\big ( d_{G({\bf p}')}(u') \big )| \leq 1$.

We claim for any distinct $u,u'$ in $U$ that 
$\mb{P}({\cal E}_{u,u'} \mid u \in U^g)
< 4\dD^{-1}|U|^{-1}$.
To see this, we condition on any choice 
of ${\bf m} = (m_w)_{w \ne u}$ such that $u$ is good. 
We let $V_0$ be the set of $v \in V$ such that
$q^u_v \notin [0.2,0.8]$, so that $|V_0| \le \dD N/2$.
For each $v \in V \sm V_0$ we have
$p_v = q^u_v + N^{-1} m_u u_v 
= q^u_v \pm N^{-1}|U| \in [0.1,0.9]$, 
so $p'_v = p_v$ for any choice of $m_u$.
Given ${\bf m}$, we can consider
$f(m_u) := {\mathbb E}\big ( d_{G({\bf p}')}(u) \big ) - 
 {\mathbb E}\big ( d_{G({\bf p}')}(u') \big )
= 	d_{G[U]}(u) - d_{G[U]}(u') 
 + ({\bf u} - {\bf u}') \cdot {\bf p}'/2$
as a function of the random variable $m_u$. 
As in the sketch above ${\mathcal E}_{u,u'}$ 
can only occur if, conditioned on $\bf m$, $f(m_u)$ 
lies in an interval $I$ of length $2$ (again, depending 
on $\bf m$). To control this probability, note that
for any $i \in [-|U|,|U|-1]$ we can write
$f(i+1) - f(i) = \sum_{v \in V} (u_v-u'_v) 
N^{-1}u_v g_{i,v}/2$, where $g_{i,v} \in [0,1]$
and $g_{i,v}=1$ for all $v \in V \sm V_0$;
the interpretation of $g_{i,v}$ is the
proportion of the total change in $p_v$
that is contained in $[0.1,0.9]$.
In particular, $f(i+1) - f(i)
\ge \sum_{v \in V \sm V_0} (u_v-u'_v)N^{-1}u_v/2
\ge N^{-1}|(N_G(u)\triangle N_G(u')) 
\cap (V \sm V_0)| \geq N^{-1}\big ( |N_G(u)\triangle N_G(u')| - |V_0| - |U|\big ) > \dD/4$.
As ${\cal E}_{u,u'}$ only occurs if $f(m_u)$ lies in the interval $I$ of length $2$,
we see that ${\cal E}_{u,u'}$ only occurs if $m_u$ lies
in an interval of length at most $8\dD^{-1}$;
the claim follows.

The conclusion is similar to that in 
the above sketch. Indeed, letting $J$ be
the graph on $U^g$ where $uu'$ is an edge
if ${\cal E}_{u,u'}$ holds, we have
$\mb{E} [e(J)] < 8\dD^{-1} |U|/2$,
so $\mb{P}(e(J)>8\dD^{-1}|U|)<1/2$.
Thus with positive probability both
$|U^g| \ge |U|/2$ and $e(J) \le 8\dD^{-1}|U|$.
By Tur\'an's theorem, $J$ has an independent set $U'$
with $|U'| \ge \dD |U|/32$, as required.
\end{proof}

\subsection{Proof of Theorems \ref{theorem: D_D_Ramsey} and  
\ref{theorem: distinct_degrees_Ramsey}}
\label{subsec:pfs}

We start with Theorem \ref{theorem: distinct_degrees_Ramsey},
which follows from Lemmas
\ref{lem: switch to probabilistic world}
and \ref{lem: constructing probability vector}. To see this, again note that by taking $c$ sufficiently small we may assume $N \geq N_0(\delta)$. Fix a $\dD$-diverse set $U$ 
of size $\tfrac{1}{2} N^{2/3}$ and set $V = V(G) \sm U$. Applying
Lemma \ref{lem: constructing probability vector} we obtain ${\bf p} \in [0.1, 0.9]^V$ 
and $U' \subset U$ with $|U'| \geq c |U|$ such that 
	$| {\mathbb E}\big ( d_{G({\bf p})}(u) \big ) - 
	{\mathbb E}\big ( d_{G({\bf p})}(u') \big ) | \geq 1$
for all distinct  $u,u' \in U'$.
Then Lemma \ref{lem: switch to probabilistic world}
gives $W \subset V$ so that $G[U \cup W]$ has at least 
$c |U'|$ distinct degrees, as required.

To deduce Theorem \ref{theorem: D_D_Ramsey}
it suffices to show that if $G$ 
is an $N$-vertex $C$-Ramsey graph 
then $G$ satisfies the hypotheses of
Theorem \ref{theorem: distinct_degrees_Ramsey},
i.e.\ has a $\dD$-diverse set $U$ 
of size $N^{2/3}$ with $\dD = \OO_C(1)$.
We can deduce this from results 
of Kwan and Sudakov \cite{KS2} as follows.
Combining their Lemma 3 part 1 and Lemma 4,
setting their $\dD$ equal to $1/4$,
we obtain $W \sub V(G)$ with $|W|=\OO_C(N)$
such that for any $u \in W$ there are
at most $|W|^{1/4}$ vertices $u' \in W$ with 
$|N_{G[W]}(u) \triangle N_{G[W]}(u')| < O_C(|W|)$.
By Tur\'an's theorem, $W$ contains an $\OO_C(1)$-diverse 
set $U$ with $|U| = \OO_C(N^{3/4})
> N^{2/3}$, as required.

\section{Optimal homogeneous sets}

In this section we will prove 
Theorem \ref{theorem: distinct_degrees_large_homogen},
which gives an optimal bound 
on $\hom(G)$ when $|V(G)| \gg f(G)$.
In the first subsection we analyse the approximate structure 
of graphs $G$ with $f(G)$ bounded. 
The second subsection introduces control graphs 
which are graphs with a special structure 
that facilitates finding induced subgraphs 
with many distinct degrees. 
The theorem itself is proved in the final subsection.

\subsection{Approximate structure}

Our first lemma, which is similar to \cite[Lemma 2.3]{BS},
shows that if a graph does not have an induced subgraph 
with many distinct degrees then we can partition its vertices
into a few parts so that vertices within any part
have similar neighbourhoods.

\begin{lem}
\label{lem: coarse structure}
Suppose that $G$ is an $N$-vertex graph with $f(G)<k$.
Then there is a partition $V(G) = V_1\cup \cdots \cup V_L$ 
with $L \leq 4k$ so that 
for all $i \in [L]$ and $u,u' \in V_i$ we have
$ \big |N_G(u) \triangle N_G(u') \big | \leq   2^{11} k^2$. 
\end{lem}

\begin{proof}
Take a maximal set $S = \{v_1,\ldots, v_L\} \subset V(G)$ 
such that $|N_G(v_i) \triangle N_G(v_j)| \geq 2^{10} k^2$ for 
all distinct $i,j$. We claim that $L\leq 4k$.
This will suffice to prove the lemma;
indeed, for any $u \in V(G)$ we can 
assign $u$ to some part $V_i$ such that 
$|N_{G}(u) \triangle N_{G}(v_i)| \leq 2^{10} k^2$,
which exists by maximality of $S$.

To prove the claim, suppose	 for contradiction
we have $S' \subset S$ with $|S'|=4k$. 
We select $W \subset V(G)$ uniformly at random 
and consider the random graph $J$ on $S'$
consisting of all pairs $\{v_i,v_j\} \sub S' \cap W$
with the same degree in $G[W]$.
Fix any $\{v_i,v_j\} \sub S'$,
write $D = |N_G(v_i)\sm N_G(v_j)|$ 
and $D' = |N_G(v_j)\sm N_G(v_i)|$,
say with $D \geq D'$. 
Conditional on any intersection 
of $W$ with $N_G(v_j)\sm N_G(v_i)$,
we can bound $\mb{P}(\{v_i,v_j\} \in J)$ by
\[ 	 \max _j{\mathbb P}(Bin(D,1/2) = j) \leq D^{-1/2} 
\leq 2|N_G(v_i)\triangle N_G(v_j)|^{-1/2}
\le (16k)^{-1}. \]
Thus $\mb{E} e(J) \le \tbinom{4k}{2} (16k)^{-1} < k/2$,
so ${\mathbb P}(e(J) \leq k) > 1/2$.
As ${\mathbb P} \big (|W\cap S'| \geq 2k\big ) \geq 1/2$, 
we can fix $W$ with $|W \cap S'| \geq 2k$ and $e(J) \leq k$. 
Tur\'an's theorem then gives $I \subset W \cap S'$
of size $k$ that is independent in $J$,
i.e.\ its vertices have distinct degrees in $G[W]$. 
This contradiction proves the claim, and so the lemma.
\end{proof}

Our next lemma shows that neighbourhood similarity
as in Lemma \ref{lem: coarse structure} implies an essentially
homogeneous graph structure between parts and within parts
(for the latter we will apply it with $V_1=V_2$).

\begin{lem}
\label{lem: close approximation to complete or empty sets}
Let $G$ be a graph with subsets $V_1$ and $V_2$ of $V(G)$ 
such that $|V_1| \ge 2D$ and
$|N_{G}(v) \triangle N_G(v')| \leq D$ 
if $\{v,v'\} \subset V_1$ or $\{v,v'\} \subset V_2$. 
Then one of the following hold:
\begin{enumerate}[(1)]
	\item each vertex in $V_1$ has at most 
	$4D$ neighbours in $V_2$, or
	\item each vertex in $V_1$ has at least 
	$|V_2| - 4 D$ neighbours in $V_2$.
\end{enumerate}
\end{lem}

\begin{proof}	
Pick $v \in V_1$ and set $A = N(v) \cap V_2$ 
and $B = V_2 \sm N(v)$. 
It suffices to show that $|A|<4D$ or $|B|<4D$. 
For each $v' \in V_1$,
neighbourhood similarity gives 
$|N_G(v') \cap A| \geq |A| - D$ 
and $|N_G(v') \cap B| \leq D$. 
Suppose for contradiction that $|A|,|B| \geq 4D$.
Then $e(V_1,A) \ge |V_1|(|A|-D) \ge |V_1| \cdot 3|A|/4$,
so there is $a \in A$ with $|N_G(a) \cap V_1| > 3|V_1|/4$.
Similarly, $e(V_1,B) \le |V_1| D \le |V_1| \cdot |B|/4$,
so there is $b \in B$ with $|N_G(b) \cap V_1| < |V_1|/4$.
However, this gives the contradiction
$|N_{G}(a) \triangle N_G(b)| > |V_1|/2 \ge D$.
\end{proof}

In combination,  Lemmas \ref{lem: coarse structure}
and \ref{lem: close approximation to complete or empty sets}
show that if $f(G)$ is bounded then $G$ has the approximate
structure of a blowup, in the sense of the next definition.
The accompanying lemma applies a merging process 
to also guarantee that this blowup is non-degenerate,
in that it is not also a blowup with fewer parts.

\begin{dfn}
Let $H$ be a graph and $\mc{P}$ be a partition of $V(H)$.
Given parts $X$, $Y$ of $\mc{P}$,
we let $H[X,Y]$ be the graph on $X \cup Y$
with edges $\{xy \in E(H): x \in X, y \in Y\}$.

We call $H$ a \emph{$\mc{P}$-blowup} 
if each such $H[X,Y]$ 
(allowing $X=Y$) is empty or complete.

We call a $\mc{P}$-blowup $H$ \emph{non-degenerate} 
if it is not also a $\mc{P}'$-blowup for some partition 
$\mc{P}'$ of $V(H)$ with fewer parts than $\mc{P}$.

We call a graph $G$ on $V(H)$ a \emph{$\DD$-perturbation}
of $H$ if for any parts $X$ and $Y$ of $\mc{P}$
and $v,v'$ in $X$ we have
$|N_G(v,Y) \triangle N_H(v',Y)|\leq \DD$.
\end{dfn}

\begin{lem}
\label{lem: coarse robust structure}
Suppose that $G$ is an $N$-vertex graph with 
a partition $(V_1,\dots,V_L)$ of $V(G)$ such that
$|N_G(u)\triangle N_G(u')| \leq D_1$ 
for all $i \in [L]$ and $u,u'$ in $V_i$. 
Let $L, T, \DD  \in {\mathbb N}$ 
with $T\geq 5\DD \geq 200L^2D_1$. 
Then there are partitions $(W,R)$ of $V(G)$
and $\mc{P}$ of $W$ such that $|R| \le LT$, 
each part of $\mc{P}$ has size at least $T$,
and $G[W]$ is a $\DD$-perturbation 
of a non-degenerate $\mc{P}$-blowup.
\end{lem}

\begin{proof}
We let $R$ be the union of all $V_i$ with $|V_i| \le T$
(so clearly $|R| \le LT$) and let $W = V(G) \sm R$. 
Next we define a partition $\mc{P}$ of $W$ by starting 
with that defined by restricting $(V_1,\dots,V_L)$ 
and repeatedly merging any two parts $X$ and $Y$
if there are some $x \in X$ and $y \in Y$
with $|N_{G[W]}(x) \triangle N_{G[W]}(y)| \leq D_2:=8LD_1$ 
(note that we measure the neighbourhood differences 
here according to $G[W]$ rather than $G$). 
This process terminates with some partition $\mc{P}$
whose parts have size at least $T$ (by definition of $R$),
so that for any distinct parts $X$, $Y$
and $x \in X$, $y \in Y$ we have
$|N_{G[W]}(x) \triangle N_{G[W]}(y)| > D_2$.

We claim that for any part $X$ of $\mc{P}$ 
we have $|N_{G[W]}(x) \triangle N_{G[W]}(x')| 
\leq L(D_1 + D_2) \leq \Delta /4$ for any $x,x'$ in $X$.
To see this, we show by induction on $t \ge 1$ 
that if $X$ is a merger of $t$ of the $V_i$'s
then $|N_{G[W]}(x) \triangle N_{G[W]}(x')| 
\leq tD_1 + (t-1)D_2$ for any $x,x'$ in $X$.
When $t=1$ this holds by our assumptions.
Now suppose $t>1$ and $X$ was obtained 
by merging $X_1$ and $X_2$ with
$|N_{G[W]}(w_1) \triangle N_{G[W]}(w_2)| \leq D_2$
for some $w_i \in X_i$.  
If each $X_i$ is a merger of $t_i$ of the $V_i$'s,
where $t=t_1+t_2$, then by induction hypothesis
$|N_{G[W]}(x_i) \triangle N_{G[W]}(x'_i)| 
\leq t_i D_1 + (t_i-1)D_2$ for any $x_i,x'_i$ in $X_i$.
Then for any  $x,x'$ in $X$ we can bound
$|N_{G[W]}(x) \triangle N_{G[W]}(x')|$ by 
$\big( t_1 D_1 + (t_1-1)D_2 \big)
+ \big( t_2 D_1 + (t_2-1)D_2 \big) + D_2 
= t D_1 + (t-1)D_2$.
This proves the claim.

It follows from Lemma 
\ref{lem: close approximation to complete or empty sets} 
that $G[W]$ is a $\DD$-perturbation 
of some $\mc{P}$-blowup $H$.
To show that $H$ is non-degenerate, we need to show that 
for any distinct parts $X$ and $Y$ of $\mc{P}$
there is some part $Z$ (possibly equal to $X$ or $Y$)
such that one of $H[X,Z]$ and $H[Y,Z]$ 
is complete and the other is empty.

To see this, we fix any $x \in X$ and $y \in Y$,
and note by the merging rule that
$|N_{G[W]}(x) \triangle N_{G[W]}(y)| > D_2 = 8LD_1$,
so there is some part $V_i$ of the original partition with
$|(N_{G[W]}(x) \triangle N_{G[W]}(y)) \cap V_i| > 8D_1$.
We must have $V_i \sub W$, 
so $|V_i| \ge T$ by definition of $R$. By Lemma 
\ref{lem: close approximation to complete or empty sets},
for any $u \in W$ we have $|N_{G[W]}(u) \cap V_i| \le 4D_1$
or $|N_{G[W]}(u) \cap V_i| \ge |V_i| - 4D_1 \ge T-4D_1$.
We deduce that one of $|N_{G[W]}(x) \cap V_i|$
and $|N_{G[W]}(y) \cap V_i|$ is $\le 4D_1$
and the other is $\ge |V_i|-4D_1$,
so they differ by at least $T-8D_1 > 2\DD$.
Let $Z$ be the part of $\mc{P}$ containing $V_i$.
As $G[W]$ is a $\DD$-perturbation of $H$,
we cannot have $H[X,Z]$ and $H[Y,Z]$ 
both complete or both empty. 
Thus $H$ is non-degenerate, as required.
\end{proof}

\subsection{Control graphs}

Our strategy for proving
Theorem \ref{theorem: distinct_degrees_large_homogen}
in the next subsection will be to find
an induced subgraph as in the next definition;
the following lemma shows that this will
indeed have an induced subgraph 
with many distinct degrees.

\begin{dfn} \label{def:control}
We call a graph $F$ a \emph{$k$-control graph} 
if there are partitions $(A,B,C)$ of $V(F)$
and $(C_1,\dots,C_t)$ of $C$,
where $A = \{a_1,\ldots, a_k\}$ 
and each $|C_i | \geq k^2 - 1$, such that
\begin{enumerate}[(i)]
\item given $(i,j) \in [k] \times [t]$ the bipartite graph $F[a_i, C_j]$ 
is either empty or complete, and
\item if $N_F(a_i) \cap C =N_F(a_j) \cap C$ and $i\neq j$ 
then $d_{F[A\cup B]} (a_i) \neq d_{F[A\cup B]} (a_j)$.
\end{enumerate}
\end{dfn}

\begin{lem}
\label{lem: control sets give distinct degrees}
If $F$ is a $k$-control graph then $f(F) \geq k$.
\end{lem}

\begin{proof}
With notation as in Definition \ref{def:control},
we randomly select integers $m_i \in [0,|C_i|]$
uniformly and independently for each $i \in [t]$,
fix $C_i' \subset C_i$ with each $|C_i'| = m_i$
and consider the induced subgraph
$F' = F[A \cup B \cup C']$ 
with $C' = \bigcup_{i \in [t]} C'_i$.
We will show that with positive probability,
the vertices in $A$ have distinct degree in $F'$,
and so $f(F) \geq f(F') \geq k$.
	
Consider any distinct $a,a'$ in $A$.	
If $N_F(a) \cap C =N_F(a') \cap C$ 
then by property (ii) we have $d_{F'}(a) \ne d_{F'}(a')$
regardless of the choice of $C_1',\ldots, C_t'$.
On the other hand, if $N_F(a) \cap C \ne N_F(a') \cap C$ 
then there is some $C_i$ such that (say) 
$C_i \sub N_F(a)$ and $C_i \cap N_F(a') = \es$.
Conditional on any choices of $\{C_j'\}_{j\neq i}$,
there is at most one choice of $m_i$ that gives
$d_{F'}(a)=d_{F'}(a')$, which occurs with probability
$(|C_i|+1)^{-1} \leq k^{-2}$. We deduce
$\mb{P}(f(F')<k) \le \tbinom{k}{2} k^{-2} < 1/2$,
so the lemma follows.
\end{proof}

In the proof of 
Theorem \ref{theorem: distinct_degrees_large_homogen}
we will obtain control graphs in each set of the partition from Lemma
\ref{lem: coarse robust structure} using the following lemma, and combine these to form a $k$-control graph.

\begin{lem} \label{lem: special case}
Let $\DD, k , n, N \in \mb{N}$ with
$n \geq 4k\DD$ and $N > (k-1)(n-1)$.

Suppose $G$ is an $N$-vertex graph
with independence number $\alpha (G) < n$
and a partition $V(G) = W \cup U$
with $|U| \leq n/2$ and
$|N_G(v) \cap W| \leq \DD$ for all $v\in V(G)$. 

Then $G$ contains a $k$-control graph with 
vertex partition $(A= \{a_1,\ldots, a_k\},B,C)$,
where $G[A]$, $G[A,C]$ are empty
and $|C| \ge |W| - k^2\DD$,
and $B$ has a partition $(B_1,\dots,B_k)$ 
with each $|B_i| = i-1$ so that 
each $G[\{a_i\}, B_j]$ 
is complete if $i=j$ or empty if $i\neq j$.
\end{lem}

\begin{proof}
If $k = 1$ then the result is clear, 
taking $a_1$ to be any vertex from $V(G)$,  
$B_1 = \emptyset $ and $C = W \setminus N_G(a_1)$.
For $k>1$ we argue by induction.
By Tur\'an's theorem, $G$ contains a vertex $a \in V(G)$ 
with degree $\Delta (G) \geq k-1$. 
Let $a_k = a$ and $B_k \sub N_G(a)$ with $|B_k|=k-1$.
Let $G'$ be obtained from $G$ by deleting
$U$, $a$ and $N_G(B_k \cup \{a\})$.
We delete at most
$1 + \Delta + (k -1)(\Delta -1) + |U| 
\leq k\Delta + n/2 \leq n-1$ vertices,
so $|V(G')| \geq N - (n-1) > (k-2)(n -1)$. 
By induction $G'$ contains
$(A'= \{a_1,\ldots, a_{k-1}\},B',C)$,
where $G[A']$, $G[A',C]$ are empty
and $|C| \ge (|W|-k\DD) - (k-1)^2\DD
\ge |W| - k^2\DD$,
and $B'$ has a partition $(B_1,\dots,B_{k-1})$ 
with each $|B_i| = i-1$ so that 
each $G[\{a_i\}, B_j]$ 
is complete if $i=j$ or empty if $i\neq j$. 
We obtain $A$, $B$ from $A'$, $B'$
by adding $a_k$, $B_k$; then $(A,B,C)$
is as required, as there are no edges 
between $B_k \cup \{a_k\}$ and $V(G')$.
\end{proof}

\begin{rem} \label{rem:simplify}
The following simplified consequence 
of Lemma \ref{lem: special case}
will often be convenient to apply.
Let $G$ be an $N$-vertex graph $G$ with $\hom(G)<n$
that is a $\DD$-perturbation of a one-part blowup
(i.e.\ a complete or empty graph).
Suppose $k = \phi(N) 
:= \lceil \tfrac{N}{n-1} \rceil \le n/4\DD$
and $N > k^2 \DD + K$ with $K \ge k^2$.
Then $G$ has a $k$-control graph
with partition $(A,B,C)$ where $|C|=K$.
\end{rem}

\subsection{Proof of Theorem \ref{theorem: distinct_degrees_large_homogen}}

To begin, we fix parameters, for reference during the proof. Set 
		\begin{align*}
			D_1= 2^{11}k^2; \quad \Delta = 2^{25}k^4; 
			\quad \Delta _1 = 2^{5}\Delta k; \quad 
			T = 2^4\Delta _1 k^2; \quad 
			n_0 = 2^9\Delta _1 k^4 = 2^{45}k^9.
		\end{align*}
	
Let $G$ be an $N$-vertex graph 
where $N = (k-1)(n-1) + 1$ and $n \geq n_0$.
We suppose for a contradiction
that $\hom (G) < n$ and $f(G) < k$.
Lemma \ref{lem: coarse structure} gives a partition 
$V(G) = V_1 \cup \cdots \cup V_L$ with $L \leq 4k$ such that 
$|N_G(u) \triangle N_G(u')| \leq  D_1$ for all $u, u' \in V_i$.

Lemma \ref{lem: coarse robust structure} then gives
partitions $(W,R)$ of $V(G)$
and $\mc{P}=(W_1,\dots,W_M)$ of $W$ such that $|R| \le LT$, 
each part of $\mc{P}$ has size at least $T$,
and $G[W]$ is a $\DD$-perturbation 
of a non-degenerate $\mc{P}$-blowup $H$.

Our aim is to find a $k$-control graph, which by Lemma 
\ref{lem: control sets give distinct degrees}
will give the required contradiction that $f(G) \ge k$. 
This control graph will have partition $(A,B,C)$
obtained by combining $k_i$-control graphs
on vertex set $E_i \sub W_i$ with partitions $(A_i,B_i,C_i)$
for each $i \in [M]$, where $\sum_i k_i = k$
and each $G[E_i,E_{i'}]$
with $i \ne i'$ is complete or empty according to $H$.
We may also need an additional $k_0$-control
graph with partition $(a_0,\es,C_0)$ 
where $k_0=1$ and $a_0 \in R$.
We will ensure that all parts $C^j_i$ 
of each $C_i$ have size at least $k^2-1$,
and the non-degeneracy of $H$ will guarantee that vertices
in distinct $A_i$'s have distinct neighbourhoods in $C$,
so this construction will indeed give 
a control graph on $(A,B,C)$.

Next we will describe an algorithm that finds 
a $k$-control graph in some cases; we will later
show how it can be modified to cover the remaining cases.

\medskip

{\bf Algorithm.}
We proceed in $M$ rounds numbered by $i \in [M]$.
At the start of round $i$ we have sets 
$W^i_j \sub W_j$ for each $j \in [M]$,
where each $W^1_j=W_j$ and we will obtain each
$W^{i+1}_j$ from $W^i_j$ by deleting 
at most $2k^2 \DD$ vertices.
As $G[W]$ is a $\DD$-perturbation of $H$, 
and $|W^i_i| \ge |W_i|-2Mk^2\DD > 2k^2\DD$,
we can apply Remark \ref{rem:simplify} 
to $G[W^i_i]$ with $K=k^2$,
thus obtaining a $k_i$-control graph 
on a set $E_i$ with partition $(A_i,B_i,C_i)$
where $|C_i| = k^2$ and $k_i = \phi(|W^i_i|) = 
\lceil \tfrac{|W^i_i|}{n-1} \rceil \le n/4\DD$.
As $G[W]$ is a $\DD$-perturbation of $H$, 
for each $j>i$ we can remove $|E_i|\DD \le 2k^2 \DD$ 
vertices from $W^i_j$ to obtain $W^{i+1}_j$
such that $G[E_i, W^{i+1}_j]$ is complete 
or empty according to $H[W_i,W_j]$.
After all rounds are complete we obtain a
$k'$-control graph with parts $(A,B,C)$
where $A = \cup A_i$, $B = \cup B_i$,
$C = \cup C_i$ and $k'=\sum k_i$.

\medskip
 
Now we consider what conditions guarantee $k'=k$ 
in the algorithm. To analyse this,
we associate vertices of $R$ with parts $W_i$ 
according to any neighbourhood similarity.
Specifically, we fix vertices $w_i \in W_i$ 
for each $i\in [M]$ and let
$$U_i := \big \{v \in R: |N_G(v,W) \triangle N_G(w_i, W)| 
 	\leq \Delta _1 \big \}.$$
We start by considering the case that 
$\cup _{i\in [M]} U_i = R$.

As $\phi$ is superadditive, we have  	
$\sum _{i\in [M]} \phi (|W_i \cup U_i|) 
\geq \phi (N) = k$. If we have
\[ \phi (|W_i \cup U_i|) 
= \phi \big (|W_i| - {4M \Delta _1 k^2}\big ) \]
for all $i$ then we deduce
\[ |A| = \sum _{i\in [M]} \phi (|W^i_i|)
 \ge \sum _{i\in [M]} 
 \phi \big (|W_i| - 4M\Delta _1k^2\big ) = 
 \sum _{i\in [M]}\phi \big (|W_i \cup U_i| \big ) \ge k.\]
Thus we can assume (possibly by relabelling) that
\[ \phi (|W_1 \cup U_1|) 
> \phi \big (|W_1| - {4M \Delta _1 k^2}\big ). \]
If $|W_1|<n/2$ we estimate
\[ |A| \ge 1 + \sum_{i\in [2,M]} 
 \phi \big (|W_i| - 4M\Delta _1k^2\big ) 
 \ge 1 + \phi(N-|R|-|W_1|)
 \ge 1 + \phi(N-(n-1)) = k.\]
Thus we can assume $|W_1| \ge n/2$.

To complete the analysis of this case,
we modify round 1 of the algorithm by 
setting $W^1_1$ equal to 
$W_1 \cup U_1$ rather than $W_1$.
By definition of $U_1$ we can apply
Lemma \ref{lem: special case} 
to either $G[W^1_1]$ or ${\ov G}[W_{1}^{1}]$,
now with $W=W_{1}$, $U = U_1$
and $2\DD_1$ in place of $\DD$,
which is valid as
$|U_1| \leq |R| \leq 4kT \leq n/2$,
and $|W^1_1| - (2\Delta _1)k^2 
\geq n/2 - (2 \Delta _1)k^2 \geq k^2$
as $n \geq 5\Delta_1k^2$.
Thus in round 1 we find a $k_1$-control graph
with $k_1 = \phi(|W_1 \cup U_1|)
> \phi \big (|W_1| - {4M \Delta _1 k^2}\big )$.
The remainder of the algorithm is the same.
Now we estimate
\[ |A| = \sum _{i\in [M]} \phi (|W^i_i|)
 \ge 1 + \sum _{i\in [M]} 
 \phi \big (|W_i| - 4M\Delta _1k^2\big ) 
 \ge 1 + \phi(N-|R|-4M^2\Delta _1k^2) \ge k.\]

It remains to consider the case
$\cup _{i\in [M]} U_i \ne R$.
Here before applying the algorithm 
we first fix $a_0 \in R \sm \big ( \cup _{i\in [M]} U_i \big )$
and choose an extra $1$-control graph $(a_0,\es,C_0)$ as follows. 
For each $i \in [M]$, by definition of $U_i$ we have
\begin{equation*}
 |N_G(a_0,W) \triangle N_G(w_i, W)| 
 > \Delta _1 = 32\Delta k \geq 4M(k^2+\Delta).
\end{equation*}
Thus we can greedily choose disjoint 
sets $C_{1,0},\ldots, C_{M,0}$ so that
each $C_{i,0}$ has size $k^2$,
is contained in some $W_{j(i)}$,
and is contained in 
$N_G(a_0,W) \setminus N_H(w_i, W)$ or 
$N_H(w_i, W) \setminus N_G(a_0,W)$
(recall that $G[W]$ is a 
$\Delta $-perturbation of $H$).
We let $C_0 = \cup _{i\in [M]} C_{i,0}$.
Then we apply the algorithm as before,
except that we now let $W_{i,1}$ be the set of
$w \in W_i\sm C_0$ with $N_G(w, C_0) = N_H(w,C_0)$,
noting that $|W_{i,1}| \geq |W_i| - (\Delta +1)|C_0|
 \geq |W_i| - 2\Delta Mk^2$, as $G[W]$ is a 
$\Delta $-perturbation of $H$.
We still obtain a control graph,
as the neighbourhood of $a_0$ differs 
from the neighbourhoods of all vertices in $E_i$ on $C_{i,0}$.
Furthermore, $|A| = |\{a_0\}| 	+ \sum _{i\in [M]}|A_i | \geq k$.
This completes the proof. \hfill \qed

\section{Concluding remarks}

This paper was concerned with the minimum possible 
value of $f(G)$ in two regimes for $\hom(G)$.
For Ramsey graphs, i.e.\ $\hom(G) = O(\log N)$,
in Theorem \ref{theorem: D_D_Ramsey}
we showed $f(G) = \OO(N^{2/3})$, 
which gives the correct order of magnitude 
(as shown by a random graph); it would be interesting 
(but no doubt very difficult) to obtain an asymptotic result.

At the other extreme, when $\hom(G)$ is large
we have obtained an exact result,
thus proving a conjecture of Narayanan and Tomon \cite{NT}.
This also makes progress on another of their
conjectures that $\hom (G) \geq N^{1/2}$ guarantees 
$f(G) = \Omega \big ( \frac {N}{\hom (G)} \big )$; indeed,
Theorem \ref{theorem: distinct_degrees_large_homogen} proves this 
in a strong form provided $\hom (G) \geq \Omega (N^{9/10})$. 
The exponent here can be reduced by taking more care with the
exceptional set $R$ in the proof, but it seems that new ideas
are needed to reduce the exponent to $1/2$.

Finally, it would be particularly interesting 
to determine the minimum order of magnitude of $f(G)$
in the intermediate range of $\hom(G)$.

\end{document}